\documentclass[12pt]{amsart}
\usepackage{graphicx}
\newtheorem{theorem}{Theorem}[section]

\newtheorem{lemma}[theorem]{Lemma}

\theoremstyle{remark}

\newtheorem{conjecture}[theorem]{Conjecture}

\numberwithin{equation}{section}

\begin{document}

\title[Non-orientable fundamental surfaces in lens spaces] {
Non-orientable fundamental surfaces in lens spaces}

\author{Miwa Iwakura and Chuichiro Hayashi}

\date{\today}

\thanks{The last author is partially supported
by Grant-in-Aid for Scientific Research (No. 18540100),
Ministry of Education, Science, Sports and Technology, Japan.}

\begin{abstract}
 We give a concrete example of
an infinite sequence of $(p_n, q_n)$-lens spaces $L(p_n, q_n)$
with natural triangulations $T(p_n, q_n)$ with $p_n$ taterahedra
such that $L(p_n, q_n)$ contains a certain non-orientable closed surface
which is fundamental with respect to $T(p_n, q_n)$ and 
of minimal crosscap number among all closed non-orientable surfaces 
in $L(p_n, q_n)$
and has $n-2$ parallel sheets of 
normal disks of a quadrilateral type
disjoint from the pair of core circles of $L(p_n, q_n)$.
 Actually, we can set $p_0=0, q_0=1, p_{k+1}=3p_k+2q_k$ and $q_{k+1}=p_k+q_k$.
\\
{\it Mathematics Subject Classification 2000:}$\ $ 57N10.\\
{\it Keywords:}$\ $
normal surface, fundamental surface, lens space, non-orientable surface,
minimal crosscap number.
\end{abstract}

\maketitle

\section{Introduction}

 The theory of normal surface was introduced 
by H. Kneser (\cite{K}) and W. Haken (\cite{H}),
and have been playing an important role 
in study of topology of $3$-manifolds.
 Almost all sorts of important surfaces,
such as essential spheres, essential tori, 
knot spanning surfaces with maximal Euler characteristics and so on, 
can be deformed to normal surfaces
and to fundamental surfaces.
 See, for example, 
\cite{JO}, \cite{JT}, \cite{St}, \cite{Sel}, \cite{HL} and \cite{L}.

 We briefly recall the definitions 
of normal surfaces and fundamental surfaces.
 In \cite{K}, K. Kneser introduced normal surfaces.
 Let $M$ be a closed $3$-manifold and $T$ a triangulation of $M$, that is, 
a decomposition of $M$ into finitely many tetrahedra, 
where all the faces of the tetrahedra of $T$ are separated into pairs,
and each pair of faces are identified in $M$.
 Let $F$ be a closed surface embedded in $M$. 
 $F$ is called a {\it normal surface} with respect to the triangulation $T$ 
if $F$ intersects each tetrahedron in disjoint union of normal disks as below, 
or in the empty set.
 There are two kinds of {\it normal disks}, 
trigons called T-disks and quadrilaterals called Q-disks. 
 Each tetrahedron contains $4$ types of T-disks and $3$ types of Q-disks 
as in illustrated in Figure \ref{fig:normaldisks}. 
 Two normal disks are of the same {\it type}
if they have vertices in the same edges of $T$ 
as in Figure \ref{fig:squarecondition}.

\begin{figure}[htbp]
\begin{center}
\includegraphics[width=8cm]{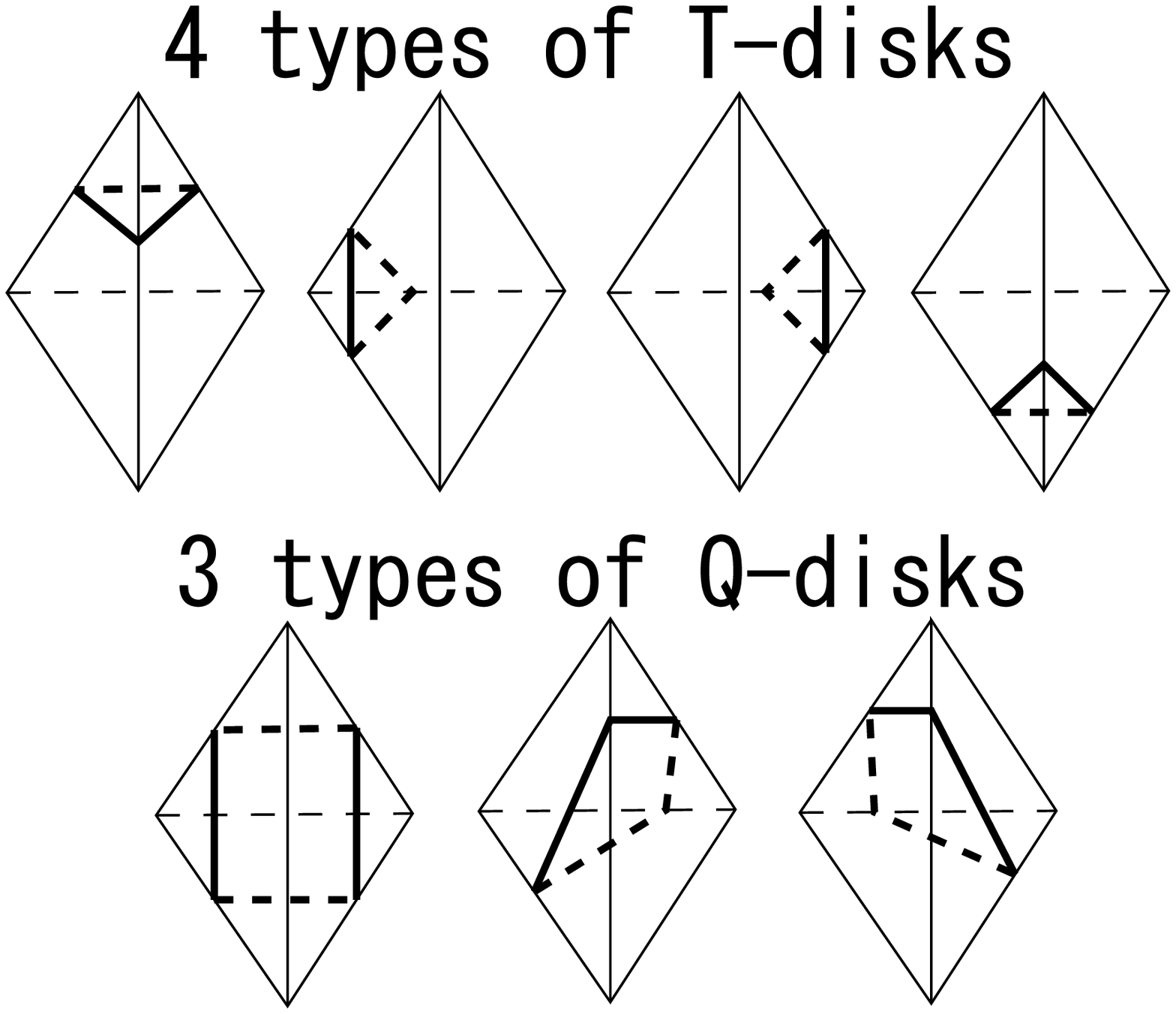}
\end{center}
\caption{}
\label{fig:normaldisks}
\end{figure}

 In \cite{H}, 
W. Haken found that normal surfaces 
correspond to non-negative integral solutions 
of a certain system of simultaneous linear equations with integer coefficients,
called the matching equations.
 First we number all the types of normal disks in the tetrahedra.
 Then we let a vector ${\bf v}_F$ represent a normal surface $F$
such that the $i$-th element $x_i$ of ${\bf v}_F$ is 
the number of the normal disks of the $i$-th type $X_i$ contained in $F$.
 In each $2$-simplex $\Delta$ of the triangulation $T$,
a properly embedded arc $\alpha$ is called a {\it normal arc}
if its two endpoints are in the interior of distinct two edges of $\Delta$.
 Two normal arcs are of the same {\it type}
if they have their endpoints in the same pair of edges of $\Delta$.
 One matching equation arises for each normal arc.
 For a type $\alpha$ of a normal arc in a $2$-simplex $\Delta$,
there are two tetrahedra $\tau_1$, $\tau_2$ which contain $\Delta$,
and each tetrahedron $\tau_i$ contains 
a trigonal type $X_{ij}$ of a normal disk
and a quadrilateral type $X_{ik}$ of a normal disk
which have an edge of type $\alpha$.
 The matching equation for $\alpha$ is $x_{1j}+x_{1k}=x_{2j}+x_{2k}$
where $x_{st}$ denotes a variable 
corresponding the number of normal disks of type $X_{st}$.
 The system of mathing equations for all the types of normal arcs 
is simply called {\it the matching equations}.

 We need some terminologies on algebra.
 Let 
${\bf v} = {}^t (v_1, \cdots, w_n)$, ${\bf w} = {}^t (w_1, \cdots, w_n)$ 
be vectors in ${\Bbb R}^n$,
where ${}^t{\bf x}$ denotes the transposition of ${\bf x}$.
 In this paper, we write ${\bf v} \le {\bf w}$ 
if $v_i \le w_i$ for all $i \in \{ 1, \cdots, n \}$.
 ${\bf v} < {\bf w}$
means that both ${\bf v} \le {\bf w}$ and ${\bf v} \ne {\bf w}$ hold.
 A vector ${\bf u} \in {\Bbb R}^n$ is {\it non-negative} 
if ${\bf 0} \le {\bf u}$,
and {\it integral} if all its elements are in ${\Bbb Z}$.

 The set of all the (possibly disconnected) normal surfaces 
is in one-to-one correspondence
with the set of all the non-negative integral solutions 
of the matching equations
satisfying the square condition as below.
 If two or three types of Q-disks exist in a single tetrahedron, 
then they intersect each other. 
 See Figure \ref{fig:squarecondition} (1).
 Hence any normal surface intersects 
each tetrahedron in Q-disks of the same type and T-disks
(Figure \ref{fig:squarecondition} (2)). 
 This is called {\it the square condition}. 

\begin{figure}[htbp]
\begin{center}
\includegraphics[width=8cm]{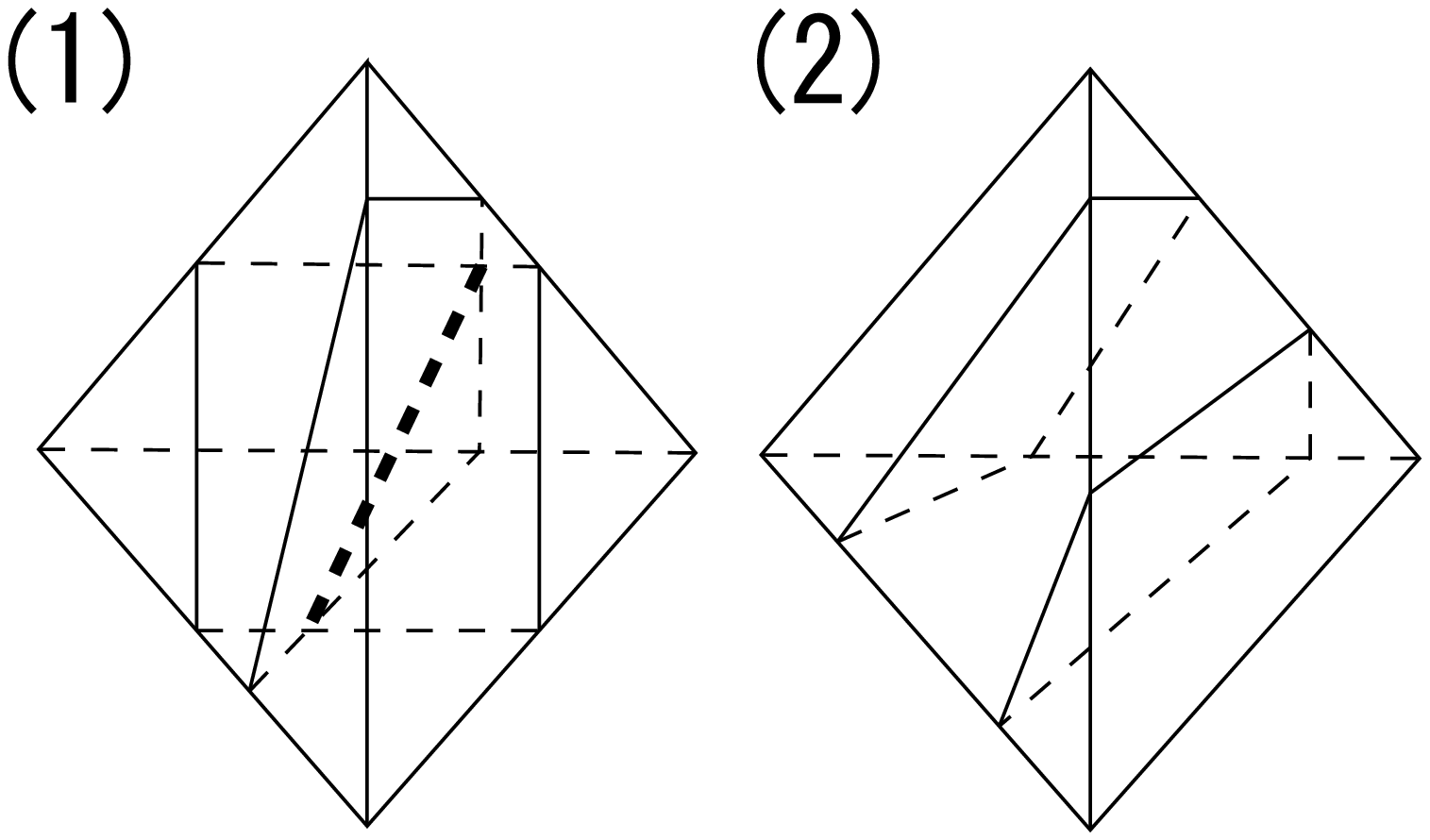}
\end{center}
\caption{}
\label{fig:squarecondition}
\end{figure}

 A normal surface is called a {\it fundamental surface}
if it corresponds to a fundamental solution
of the system of the matching equations
which is defined as below. 
 Let $A {\bf x}={\bf 0}$ be a linear system of equations, 
where $A$ is a matrix with all the elements in ${\Bbb Z}$, 
and ${\bf x}$ is a vector of variables. 
 $V_A$ denotes the solution space of the linear system 
considered in ${\Bbb R}^n$. 
 A non-zero non-negative integral solution ${\bf v}$ 
is called a {\it fundamental solution}, 
if there is no integral solution ${\bf v}' \in V_A$ 
with ${\bf 0} < {\bf v}' < {\bf v}$.
 There are only finite number of fundamental solutions for each system.
 Moreover, an upper bound for elements of fundamental solutions is known
(\cite{HLP}).
 Hence there is an algorithm which determines all the fundamental solutions,
though it is not practical.

\begin{figure}[htbp]
\begin{center}
\includegraphics[width=6cm]{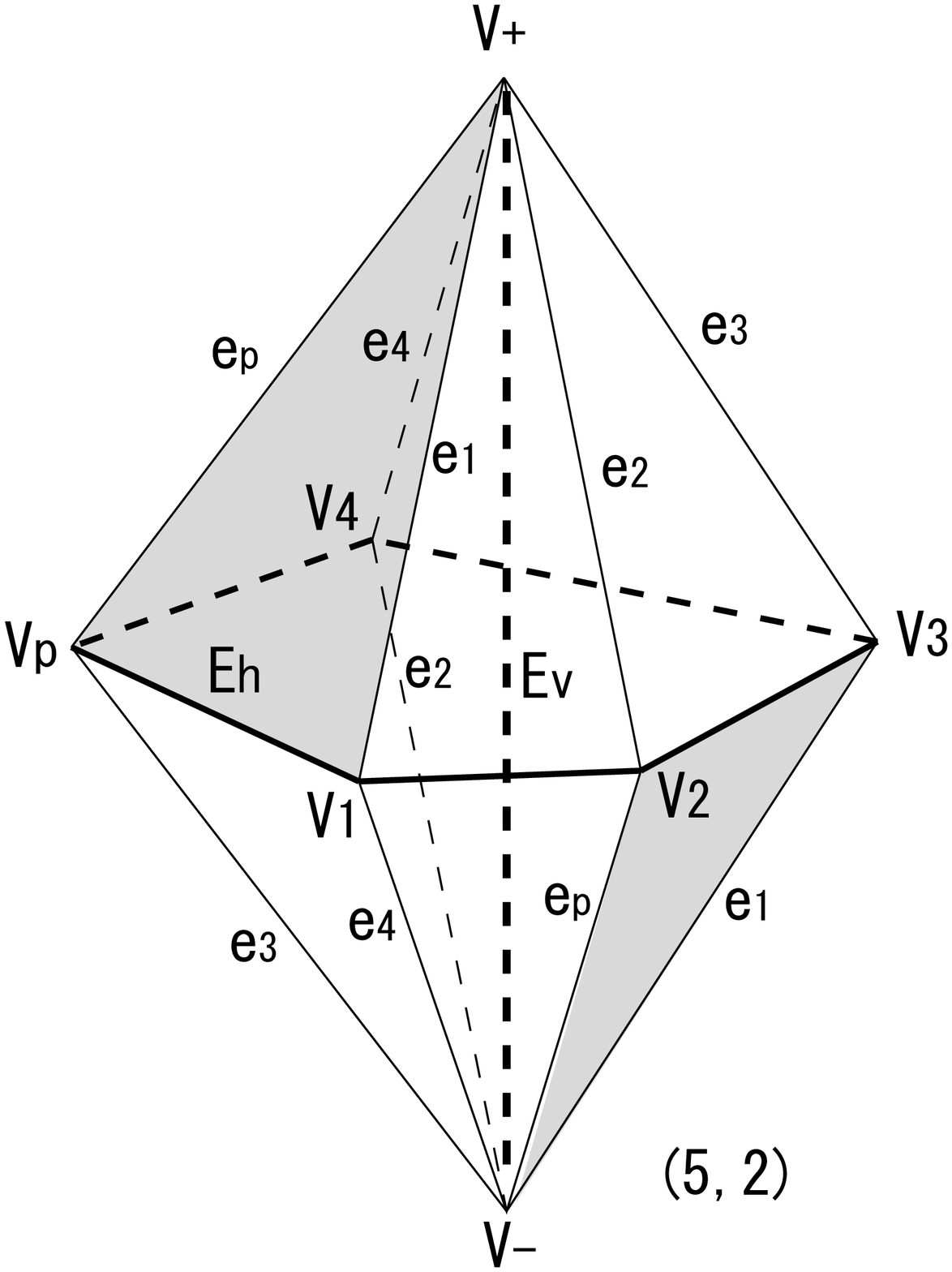}
\end{center}
\caption{}
\label{fig:suspension}
\end{figure}

 Let $p$ and $q$ be positive integers
such that $p$ and $q$ are coprime.
 We can obtain a {\it $(p,q)$-lens space} $L(p,q)$
from a suspension of a $p$-gon 
by gluing each trigonal face in the upper hemisphere 
with that in the lower hemisphere, 
performing ($2\pi q/p$)-rotation 
and taking a mirror image about the equator. 
 See Figure \ref{fig:suspension}.
 Precisely, the trigon $v_+ v_i v_{i+1}$ is glued to $v_- v_{i+q} v_{i+q+1}$, 
where indices are considered modulo $p$. 
 The edge $e_i$ connects $v_+$ and $v_i$, and also $v_-$ and $v_{i+q}$. 
 The horizontal edges connecting $v_i$ and $v_{i+1}$ for $1 \le i \le p$
are all glued up together into an edge $E_h$. 
 Taking an axis $E_v$ connecting the vertices $v_+$ and $v_-$ 
in the suspension, 
we can decompose it into $p$ tetrahedra. 
 This gives a natural triangulation $T(p,q)$ of a $(p,q)$-lens space. 
 The $i$-th tetrahedron $\tau_i$ has 
vertices $v_+$, $v_-$, $v_i$ and $v_{i+1}$. 

 An embedded circle $C$ in a lens space $M$
is called a {\it core} circle
if the exterior $M-$int\,$N(C)$ is homeomorphic to the solid torus.
 Note that each of $E_h$ and $E_v$ forms a core circle of $L(p,q)$.

\begin{figure}[htbp]
\begin{center}
\includegraphics[width=10cm]{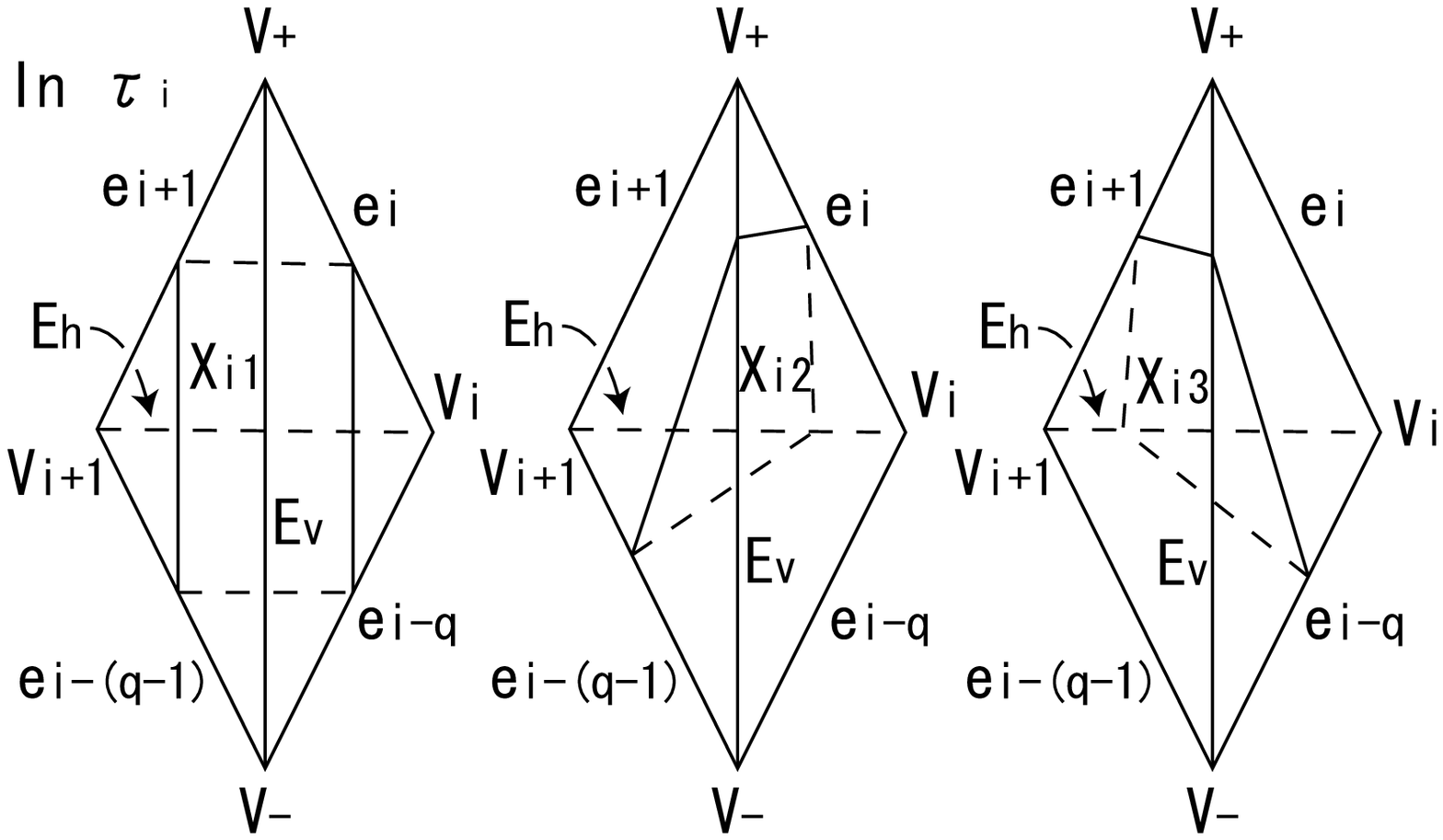}
\end{center}
\caption{}
\label{fig:numberQdisk}
\end{figure}

 In lens spaces, 
non-orientable closed surfaces with maximum Euler characteristics 
are interesting.
 A formula for calculating the maximum Euler characteristic
is given by G. E. Bredon and J. W. Wood in \cite{BW}.

\begin{theorem}\label{theorem:NSheets}
 Let $\{ p_n \}$, $\{ q_n \}$ be infinite sequences of integers
such that $p_0 = 0, q_0 = 1, p_{k+1} = 3p_k + 2q_k$ and $q_{k+1} = p_k + q_k$
for any non-negative integer $k$.
 For $n \ge 2$,
the $(p_n, q_n)$-lens space
contains 
a non-orientable closed surface ${\bf h}_{n-1}$
with maximal Euler characteristic
which is fundamental with respect to Haken's matching equations 
on the triangulation $T(p_n, q_n)$
and homeomorphic to the connected sum of $n$ projective planes
and has $n-2$ sheets of quadrilateral normal disks of type $X_{m1}$
shown in Figure \ref{fig:numberQdisk}
for some $m$.
 The construction of the fundamental surface ${\bf h}_{n-1}$
is descrived in section \ref{section:construction}.
\end{theorem}

 In \cite{HST}, Hass, Snoeyink and Thurston gave
an example of infinite sequence of polygonal knots $K_n$ in ${\Bbb R}^3$
such that $K_n$ has fewer than $10n+9$ edges
and its piecewise linear triangulated disk spanning $K_n$
contains at least $2^{n-1}$ flat triangles.
 This implies existence of fundamental surfaces 
with huge number of sheets of some type of a normal disk.
 However, no concrete example of such surface was given.

 Fominykh gave a complete description of fundamental surfaces
with respect to certain good handle decompositions
for three manifolds including lens spaces in \cite{F}.
 However, no fundamental surface there has a $2$-handle
with weight more than two.

\begin{conjecture}
 The surface ${\bf h}_{n-1}$, 
which will be constructed in section \ref{section:construction}
and is fundamental with respect to the Haken's matching equations, 
is fundamental also with respect to the Q-matching equations
on the triangulation $T(p_n, q_n)$.
\end{conjecture}

 For the definition of Q-matching equation, see \cite{T}.


\section{Preliminaries}

 In this section, 
we introduce two basic lemmas
on closed non-orientable surfaces in lens spaces.

 The next lemma is well-known.
 See ll.24-28 in p.97 in \cite{BW}.

\begin{lemma}\label{lemma:orientability}
 Let $M$ be a lens space, and $C$ a core of it.
 Let $F$ be a closed surface $F$ embedded in $M$
intersecting $C$ transversely.
 Then $F$ is non-orientable 
if and only if it intersects $C$ in odd number of points.
\end{lemma}

\begin{lemma}\label{lemma:HakenFund}
 Let $F$ be a normal surface 
in the $(p,q)$-lens space
with the triangulation $T(p,q)$.
 If $F$ intersects each of $E_v$ and $E_h$ in a single point
and contains a normal disk of type $X_{k2}$ or $X_{k3}$,
then $F$ is fundamental with respect to Haken's matching equations.
 In particular, $F$ is connected.
\end{lemma}

\begin{proof}
 Suppose, for a contradiction,
that $F$ is decomposed as $F = F_1 + F_2$.
 Since $F$ intersects $E_v$ and $E_h$ in a single point,
one of $F_1$ or $F_2$, say $F_1$ intersects $E_v$ and $E_h$ in a single point,
and $F_2$ is disjoint from $E_v \cup E_h$.
 Since $X_{i1}$ is the only type of normal disk in $\tau_i$
which is disjoint from $E_v \cup E_h$ for all $i \in \{ 1, 2, \cdots, p\}$,
the surface $F_2$
intersects each tetrahedron $\tau_i$
in copies of the quadrilateral $X_{i1}$.
 Hence $F_2$ is a union of $n$ parallel copies 
of the Heegaard splitting torus surrounding $E_v$ and $E_h$
for some positive integer $n$.
 $F_2$ intersects each $\tau_i$ in $n$ sheets of normal disks
of type $X_{i1}$.
 This implies that $F \cap \tau_i$
contains a normal disk of type $X_{i1}$ for each $i \in \{ 1, 2, \cdots, p\}$.
 By assumption, 
for some $k \in \{ 1, 2, \cdots, p \}$
$F$ has a normal disk of type $X_{k2}$ or $X_{k3}$
which cannot exist together with a normal disk of $X_{k1}$
by the square condition.
 This is a contradiction.
\end{proof}


\section{Construction of surfaces}\label{section:construction}

 In this section, we construct the fundamental surface ${\bf h}_{n-1}$ 
in Theorem \ref{theorem:NSheets}.

 Tollefson introduced Q-coordinates representing normal surfaces in \cite{T}. 
 We first number types of Q-disks. 
 In our case, 
we number the $3$-types of Q-disks $i1$, $i2$ and $i3$ in $\tau_i$
as in Figure \ref{fig:numberQdisk}, 
where the axis $E_v$ is in front of the tetrahedron. 
 $X_{i1}$ separates the edges $E_h$ and $E_v$,
$X_{i2}$ separates $e_{i+1}$ and $e_{i-q}$
and $X_{i3}$ does $e_i$ and $e_{i-(q-1)}$.


 Let $x_{ij}$ be the number of Q-disks of type $X_{ij}$
contained in a normal surface $F$. 
 Then we place them in a vertical line 
in the order of indices lexicographically, 
to obtain a {\it Q-coordinate} 
${\bf v}_Q (F) = {}^t (x_{11}, x_{12}, x_{13}, x_{21}, x_{22}, x_{23}, 
x_{31}, \cdots, x_{t1}, x_{t2}, x_{t3} )$ of $F$,
where $t$ is the number of tetrahedra of the triangulation. 
 A normal surface with the Q-coordinate ${\bf 0}$
is called {\it trivial}.
 It is composed of trigonal normal disks
and has no quadrilateral normal disks,
and hence is 
a disjoint union of $2$-spheres
each of which surrounds a vertex of the triangulation.
 Any non-zero solution of the system of Q-matching equations
with non-negative integer elements
determines a normal surface with no trivial component uniquely.
 In addition, for any normal surface $F$ with no trivial component,
and for any normal surface $F'$ with ${\bf v}_Q (F') = {\bf v}_Q (F)$,
there is a trivial normal surface $\Sigma$
such that $\Sigma$ is disjoint from $F$ 
and $F'= F + \Sigma$
(Theorem 1 in \cite{T}).
 He introduced there the Q-matching equations
on numbers of sheets of normal disks of Q-disk types
such that 
the set of non-trivial non-negative integral solutions 
satisfying the square condition
is in one to one correspondence
with the set of normal surfaces with no trivial component in the $3$-manifold.

 In our situation of $(p,q)$-lens space, 
the $(3i-2)$-nd, $(3i-1)$-st and $(3i)$-th elements 
$x_{i1}, x_{i2}, x_{i3}$ of a Q-coordinate 
together form the $i$-th {\it block} for $1\le i \le p$. 
 We often put a line $\lq\lq |$" instead of a comma
between every adjacent pair of blocks
in such a manner as 
${}^t (x_{11}, x_{12}, x_{13} \ |\ x_{21}, x_{22}, x_{23} 
\ |\ \cdots \ |\ x_{p1}, x_{p2}, x_{p3})$.

 In the previous paper,
we have shown the next lemma.

\begin{lemma}\label{lemma:generator} {\rm (Lemma 1.3 in \cite{IH})}
 For the triangulation $T(p,q)$ of the $(p,q)$-lens space
with $p \ge 5$, $2 \le q < p/2$ and $GCM(p,q)=1$, 
the vectors
${\bf s}_1, {\bf s}_2, \cdots, {\bf s}_p$,
${\bf t}_1, {\bf t}_2, \cdots, {\bf t}_p$ as below
form a basis of the solution space in ${\Bbb R}^{3t}$
of the system of Q-matching equations.
\newline
 $($The $j$-th block of ${\bf s}_i) = \left\{ 
\begin{array}{l}
{}^t(1,1,1)\ \ {\rm if}\ j=i \\
{}^t(0,0,0)\ \ {\rm otherwise}
\end{array} \right.$
\newline
 $($The $j$-th block of ${\bf t}_i)= \left\{ 
\begin{array}{l}
{}^t(0,1,0)\ \ {\rm if}\ j=i\ {\rm or}\ i+q+1 \\
{}^t(0,0,1)\ \ {\rm if}\ j=i+1\ {\rm or}\ i+q \\
{}^t(0,0,0)\ \ {\rm otherwise}
\end{array} \right.$

 Hence a general solution ${\bf v}$ in ${\Bbb R}^{3t}$
is presented as below. 
\newline
${\bf v}=a_1 {\bf s}_1 + \cdots +a_p {\bf s}_p 
+ b_1 {\bf t}_1 + \cdots + b_p {\bf t}_p$
\newline
$= 
{}^t (a_1,\ a_1+b_1+b_{p-q},\ a_1+b_p+b_{p-q+1} \ | \cdots |\ 
a_i,\ a_i+b_i+b_{p-q+i-1},\ a_i+b_{i-1}+b_{p-q+i} \ | \cdots $
\newline
\hspace*{2cm}
$\cdots |\ a_p,\ a_p+b_p+b_{p-q-1},\ a_p+b_{p-1}+b_{p-q})$,
\newline
with $a_1, \cdots, a_p, b_1, \cdots, b_p\in {\Bbb R}$,
where 
${}^t (a_i,\ a_i+b_i+b_{p-q+i-1},\ a_i+b_{i-1}+b_{p-q+i})$
is the $i$-th block.
\end{lemma}

 When $p$ is even and $q \ge 3$,
in the $(p,q)$-lens space
the normal surface represented by

$${\bf h}_0 = (\sum_{m=1}^{p/2} {\bf t}_{2m-1})/2 = 
{}^t (0 \ 1 \ 0 | 0 \ 0 \ 1 | 0 \ 1 \ 0 | 0 \ 0 \ 1 |
\cdots | 0 \ 1 \ 0 | 0 \ 0 \ 1 )$$ 

\noindent
is non-orientable by Lemma \ref{lemma:orientability}
and fundamental by Lemma \ref{lemma:HakenFund}.
 In fact, this surface 
consists of only Q-disks of types $X_{k2}$ and $X_{\ell 3}$
for any odd integer $k$ and any even integer $\ell$,
and intersects each of $E_h$ and $E_v$ in a single point.
 However, ${\bf h}_0$ is not Q-fundamental 
because it is larger than ${\bf t}_1$.
 In fact, it is of Euler characteristic $2 - p/2$,
which is not maximal among all closed non-orientable surfaces 
in a $(p,q)$-lens space with $p$ even and $q \ge 3$
by \cite{BW}.
 We will perform compressing operations 
on the surface represented by ${\bf h}_0$
to obtain one of maximal Euler characteristic.

 In what follows,
we consider the infinite sequence of $(p_n, q_n)$-lens spaces,
where $p_0=0$, $q_0=1$,
$p_n = 3 p_{n-1} + 2 q_{n-1}$
and $q_n = p_{n-1} + q_{n-1}$,
which are derived from
$\dfrac{p_n}{q_n} = 
2+\dfrac{1}{1+\dfrac{1}{(\dfrac{p_{n-1}}{q_{n-1}})}}$
and $\dfrac{p_1}{q_1} = \dfrac{2}{1}$.
 Thus $(p_1,q_1) = (2,1), \ (p_2,q_2) = (8,3),\ 
(p_3,q_3)=(30,11), \ (p_4,q_4)=(112, 41), \ 
(p_5,q_5)=(418,153), \ (p_6,q_6)=(1560,571), \cdots$.
 Note that $p_n$ is even for any natural number $n$.
 We will obtain some formulae 
on this infinite sequences $\{ p_n \}$, $\{ q_n \}$
in Lemma \ref{lemma:formulae} 
in section \ref{section:formulae}.

 In the case of $n \ge 2$,
we perform $n-1$ steps of compressing operations on ${\bf h}_0$
to obtain a desired closed non-orientable surface ${\bf h}_{n-1}$ of maximal Euler characteristic.
 The $k$-th step is composed of $(q_{(n-(k-1))} - 1)/2$ compressing operations.


 For $1 \le i \le (q_{n}-1)/2$,
the $i$-th compressing disk of the first step
is placed in the union of four tetrahedra 
$\tau_{2i-1}, \tau_{2i}, \tau_{q_n + 2i -1}$ and $\tau_{q_n + 2i}$.
 It intersects each tetrahedron in a single trigonal disk patch.
 We call the pair of patches in $\tau_{2i-1} \cup \tau_{2i}$ 
the first pair (see Figure \ref{fig:TCprDiskPatch} (a)),
and that in $\tau_{q_n + 2i-1} \cup \tau_{q_n+2i}$ 
the second (see Figure \ref{fig:TCprDiskPatch} (b)).
 Each pair is composed of 
the leading trigonal disk patch in the tetrahedron 
assigned the smaller number ($\tau_{2i-1}, \tau_{q_n+2i-1}$)
and the following one in that assigned the larger number
($\tau_{2i}, \tau_{q_n+2i})$.

\begin{figure}[htbp]
\begin{center}
\includegraphics[width=10cm]{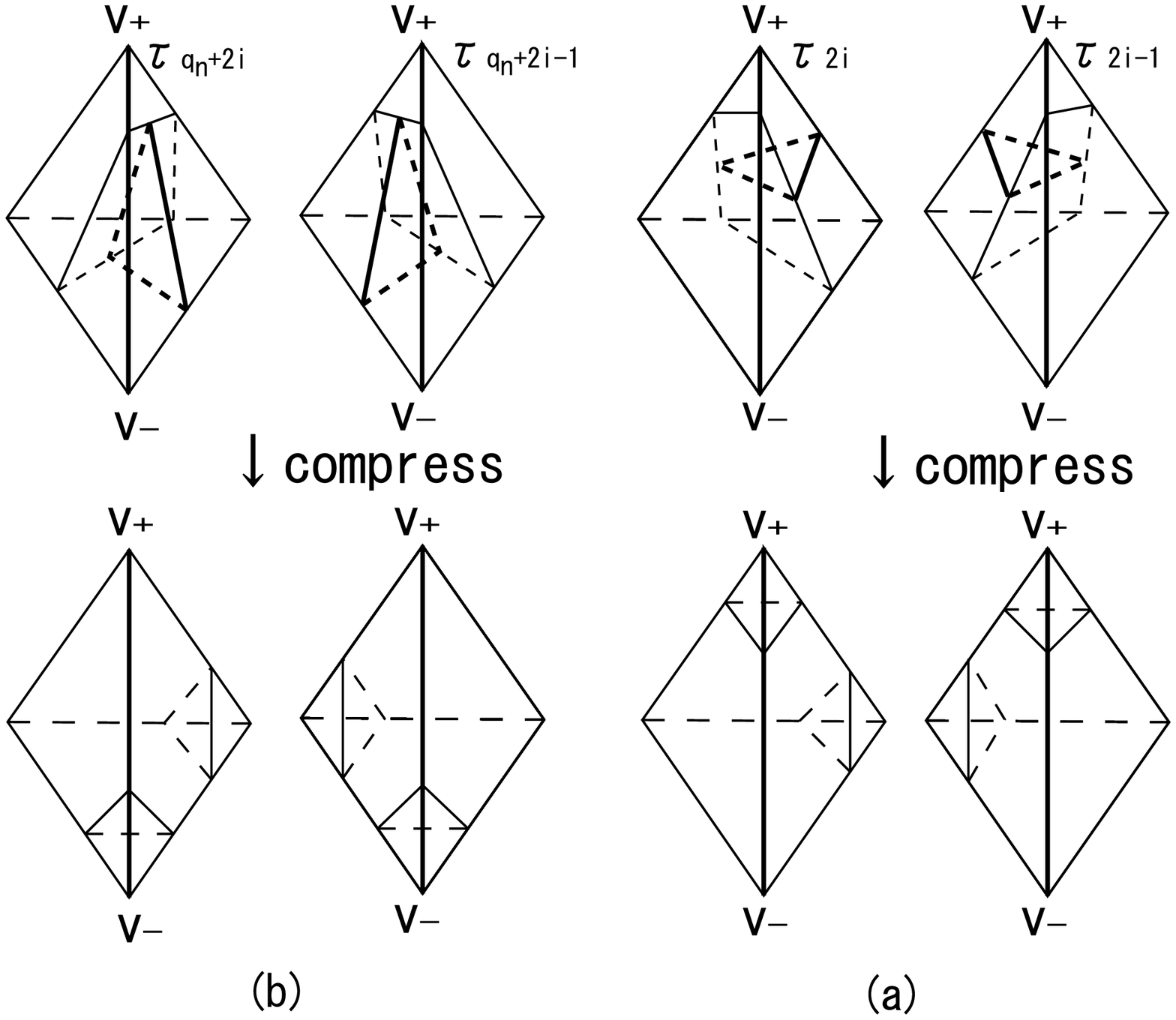}
\end{center}
\caption{}
\label{fig:TCprDiskPatch}
\end{figure}


 The compressions of the first step deform the Q-coordinate ${\bf h}_0$ into

$${\bf h}_1 = {\bf h}_0 - \sum_{i=1}^{(q_n -1)/2} {\bf t}_{2i-1}.$$

\noindent
Subtracting ${\bf t}_{2i-1}$ corresponds 
to the surgery along the $i$-th compressing disk.
 In fact, the four Q-disks of ${\bf t}_{2i-1}$ are deformed to T-disks 
by the $i$-th surgery as shown in Figure \ref{fig:TCprDiskPatch}.
 A single compression increases the Euler characteristic of the surface by two.
 Hence we obtain a closed surface 
of Euler characteristic $2 - p_n /2 + q_n -1$.
 The resulting surface is non-orientable 
and fundamental with respect to Haken's matching equations
by Lemmas \ref{lemma:orientability} and \ref{lemma:HakenFund}.

 See Figure \ref{fig:ComprDisk8-3} 
where described the compressing disks for the case of $(p_2, q_2) = (8,3)$.
 We obtain a closed non-orientable surface as in Figure \ref{fig:NonOriIn83}
after the compression in the first step.
 This is a Klein bottle, 
which is of maximal Euler characteristic
among all closed non-orientable surfaces in the lens space $L(8,3)$.

\begin{figure}[htbp]
\begin{center}
\includegraphics[width=7cm]{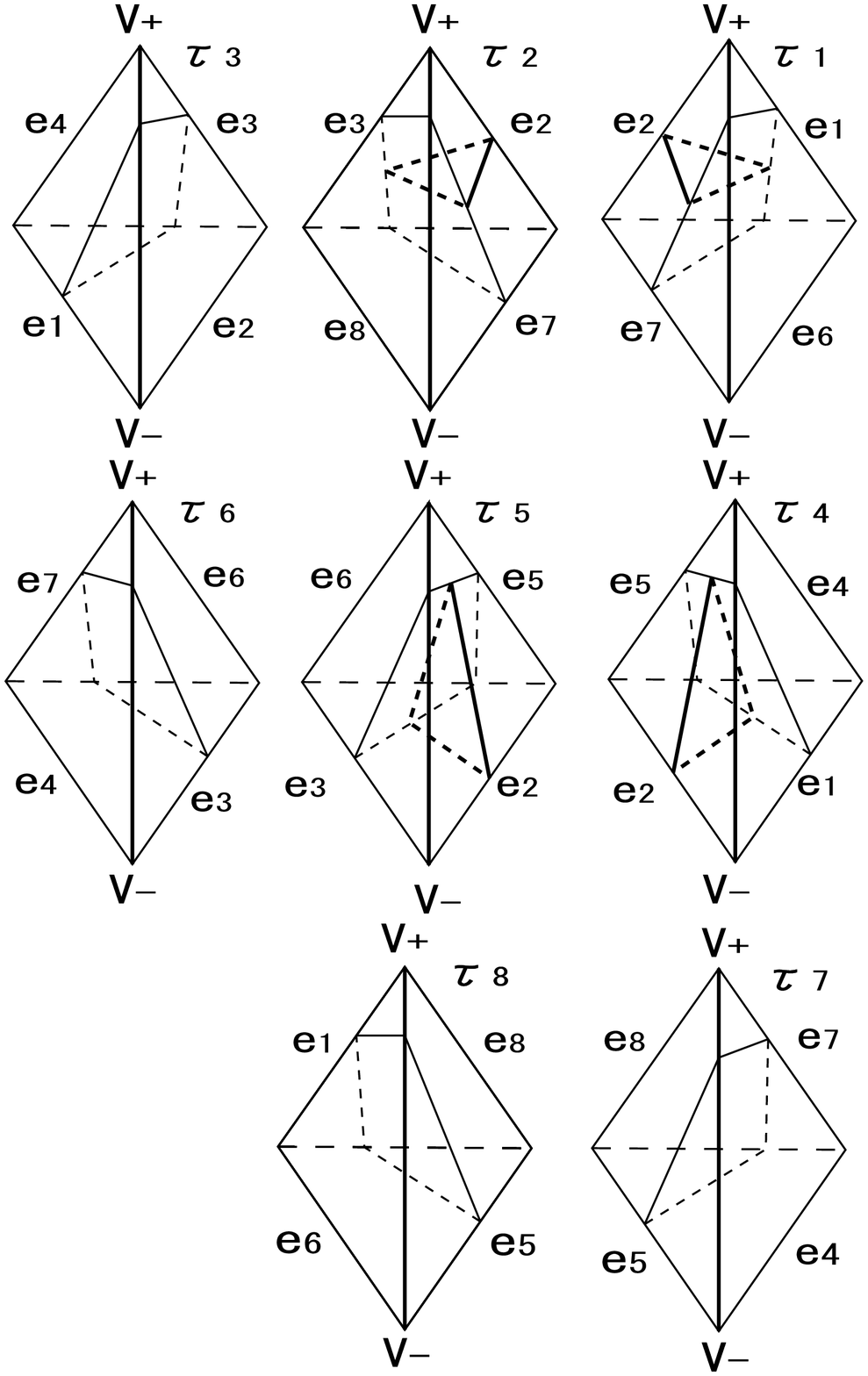}
\end{center}
\caption{}
\label{fig:ComprDisk8-3}
\end{figure}

\begin{figure}[htbp]
\begin{center}
\includegraphics[width=7cm]{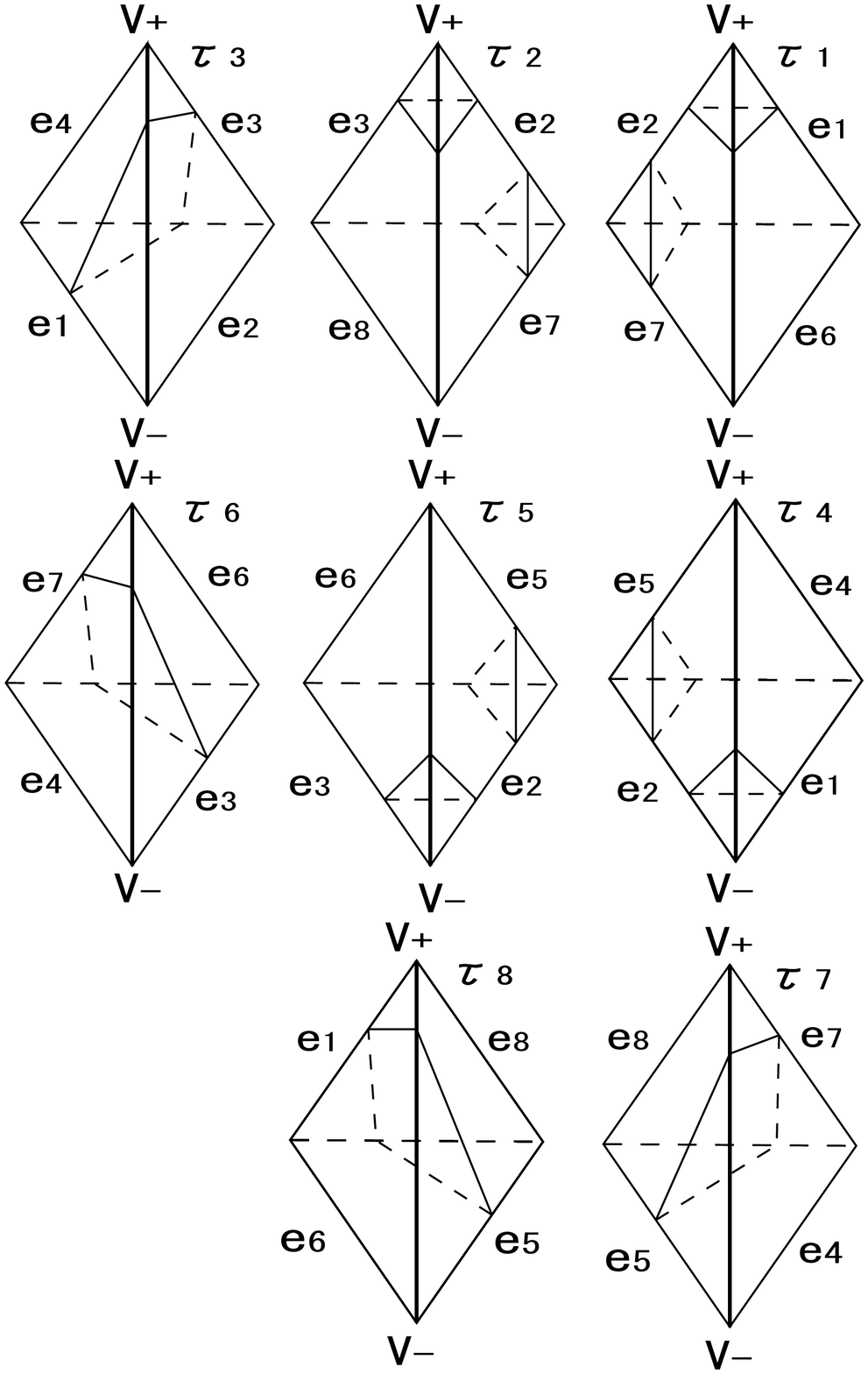}
\end{center}
\caption{}
\label{fig:NonOriIn83}
\end{figure}


 We call the union of the first $q_n$ tetrahedra $\tau_1, \tau_2, \cdots, \tau_{q_n}$
and the union of corresponding $q_n$ blocks of the Q-coordinate
{\it the first region},
the union of the second $q_n$ tetrahedra $\tau_{q_n + 1}, \tau_{q_n + 2}, \cdots, \tau_{2q_n}$
and the union of corresponding $q_n$ blocks
{\it the second region},
the union of remainder $p_{n-1}$ tetrahedra $\tau_{2q_n+1}, \tau_{2q_n+2}, \cdots, \tau_{p_n}$
and the union of corresponding $p_{n-1}$ blocks
{\it the last region}.
 In fact, the last region is composed of 
$p_n - 2q_n = (3p_{n-1}+2q_{n-1}) - (p_{n-1}+q_{n-1}) = p_{n-1}$ tetrahedra or blocks.

 Compressing operations of the first step have occurred
in all the tetrahedra corresponding to the first and the second regions
but in the ending tetrahedra $\tau_{q_n}$ and $\tau_{2q_n}$.


 When $n \ge 3$,
we perform the second step of compressions.
 This operation is performed 
along $(q_{n-1}-1)/2$ compressing disks.
 The $i$-th compressing disk is composed of four pairs of disk patches.
 The first pair consists of two trigonal disk patches 
in $\tau_{2q_n + 2i - 1} \cup \tau_{2q_n + 2i}$ 
similar as those described in Figure \ref{fig:TCprDiskPatch} (a),
the second pair consists of quadrilateral disk patches 
in $\tau_{3q_n + 2i -1} \cup \tau_{3q_n + 2i}$
as described in Figure \ref{fig:QCprDiskPatch} (a),
the third pair quadrilateral disk patches
in $\tau_{4q_n + 2i -1} \cup \tau_{4q_n + 2i}$
as described in Figure \ref{fig:QCprDiskPatch} (b),
and the last pair trigonal disk patches 
in $\tau_{5q_n + 2i -1} \cup \tau_{5q_n + 2i}$
similar as those described in Figure \ref{fig:TCprDiskPatch} (b).
 Because $2q_n + 2i -1 > 2q_n$ and 
$2q_n + 2i \le 2q_n + 2 \cdot (q_{n-1}-1)/2 
= 2q_n + q_{n-1} < 2q_n + p_{n-1} = p_n$,
the first pair of two trigonal disk patches are in the last region,
where compressions of the first step have not been performed.
 Since 
\newline
$3q_n+2i-1 
= 3(p_{n-1}+q_{n-1})+2i-1
= (3p_{n-1}+2q_{n-1})+q_{n-1}+2i-1$
\newline
$= p_n + q_{n-1} + 2i -1
\equiv q_{n-1}+2i-1$ (mod. $p_n$),
\newline
the second pair of quadrilateral disk patches are 
in  $\tau_{(q_{n-1}+2i-1)} \cup \tau_{(q_{n-1}+2i)}$.
 In these tetrahedra, compressions of the first step have already occurred.
 In fact, 
$q_{n-1} + 2i -1 \ge 1$ and
$q_{n-1} +2i \le q_{n-1} + 2 \{ (q_{n-1}-1)/2 \} 
= 2q_{n-1}-1 < p_{n-1} + q_{n-1} = q_n$,
and hence the tetrahedra are in the first region,
and are not the ending ones.
 The third pair of quadrilateral disk patches are contained
in  the $(q_{n-1}+2i-1)$-st and the $(q_{n-1}+2i)$-th tetrahedra
in the second region, 
where compressions of the first step have been performed.
 The last pair of trigonal patches are
in the $(q_{n-1}+2i-1)$-st and the $(q_{n-1}+2i)$-th tetrahedra in the last region,
i.e., in $\tau_{2q_n + q_{n-1} + 2i -1} \cup \tau_{2q_n + q_{n-1} + 2i}$, 
where compressions of the first step have not occurred.
 Moreover, the first pairs of disk patches 
and the last ones do not appear in the same tetrahedron
because 
$2q_n + q_{n-1} + 2i' - 1 \ge 2q_n + q_{n-1} + 1 > 2q_n + q_{n-1} -1 \ge 2q_n + 2i$
for the leading one of the last pair of disk patches 
of the $i'$-th compressing disk
and the following one of the $1$st pair of disk patches 
of the $i$-th compressing disk.

\begin{figure}[htbp]
\begin{center}
\includegraphics[width=10cm]{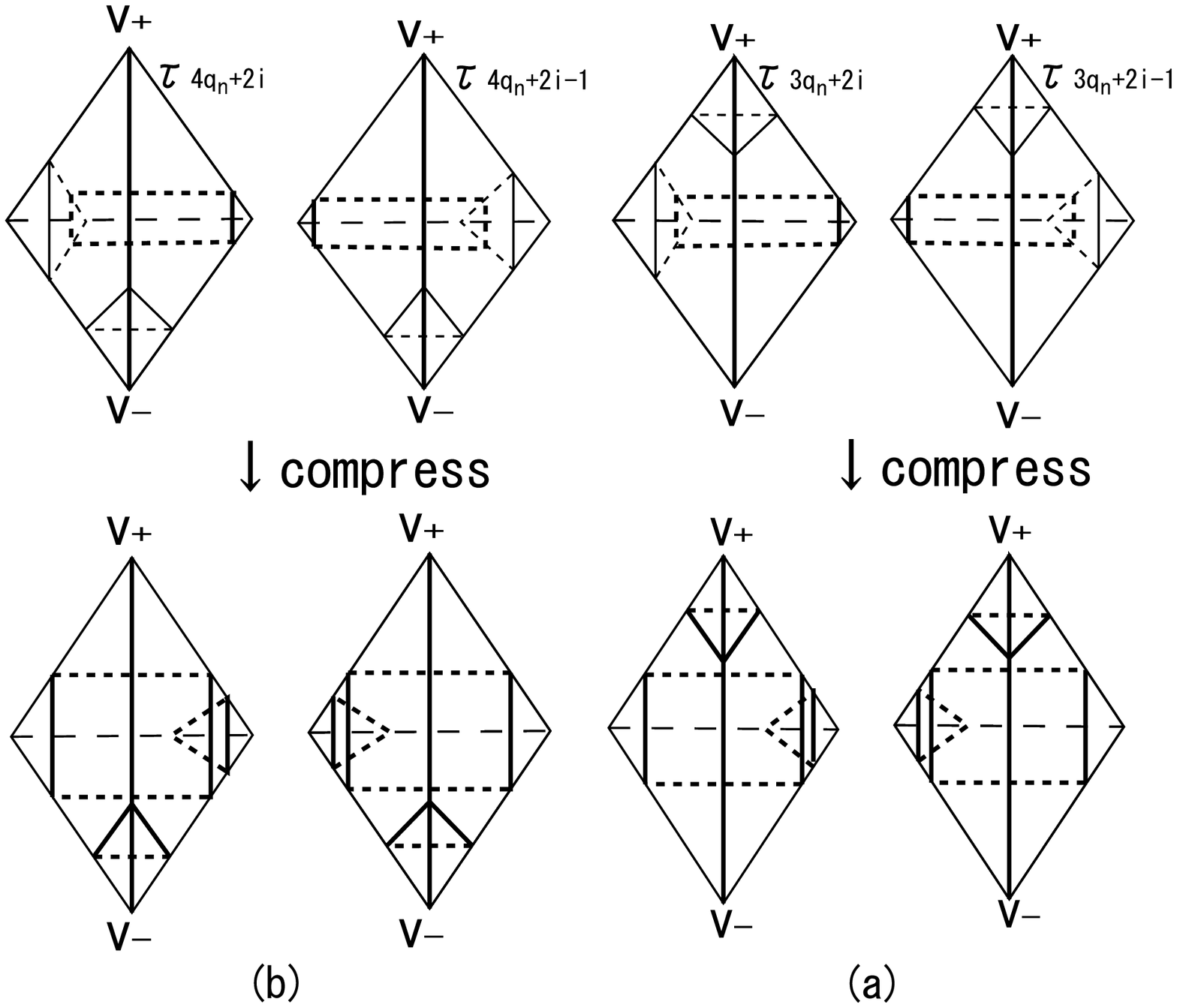}
\end{center}
\caption{}
\label{fig:QCprDiskPatch}
\end{figure}

 See Figure \ref{fig:CprDisk30-11}
where described the compressing disks of the $2$nd step 
for the case of $(p_3, q_3) = (30,11)$.

 After the second step Q-coordinate will be

$${\bf h}_2 = 
{\bf h}_1 
-\sum_{i=1}^{(q_{n-1} -1)/2} 
({\bf t}_{2q_n + 2i-1} + {\bf t}_{3q_n + 2i-1} + {\bf t}_{4q_n + 2i-1}
- {\bf s}_{3q_n + 2i -1} - {\bf s}_{3q_n + 2i}
- {\bf s}_{4q_n + 2i-1} - {\bf s}_{4q_n + 2i}).$$

 Subtracting 
${\bf t}_{2q_n + 2i-1} + {\bf t}_{3q_n + 2i-1} + {\bf t}_{4q_n + 2i-1}
- {\bf s}_{3q_n + 2i -1} - {\bf s}_{3q_n + 2i}
- {\bf s}_{4q_n + 2i-1} - {\bf s}_{4q_n + 2i}$
corresponds to the surgery along the $i$-th compressing disk.
 In fact,
after surgery on the $i$-th compressing disk,
four Q-disks in $\tau_{2q_n+2i-1} \cup \tau_{2q_n+2i} \cup \tau_{5q_n+2i-1} \cup \tau_{5q_n+2i}$
are changed to T-disks 
by deformations along the first and the last pairs of trigonal disk patches.
 A Q-disk disjoint from $E_h \cup E_v$ is added
in 
$\tau_{3q_n+2i-1} \cup \tau_{3q_n+2i} 
\cup \tau_{4q_n+2i-1} \cup \tau_{4q_n+2i}$
by deformations along the second and the third pairs 
of quadrilateral disk patches.
 More precisely, 
trigonal normal disks on the right hand sides of tetrahedra
are replaced by ones on the left hand sides,
and vice versa.
 See Figure \ref{fig:QCprDiskPatch}.

 The resulting surface is non-orientable 
and fundamental with respect to Haken's matching equations
by Lemmas \ref{lemma:orientability} and \ref{lemma:HakenFund}.
 A single compression increases the Euler characteristic of the surface by two.
 Hence we obtain a closed surface 
of Euler characteristic $2 - p_n /2 + q_n -1 + q_{n-1} -1$.

\begin{figure}[htbp]
\begin{center}
\includegraphics[width=12cm]{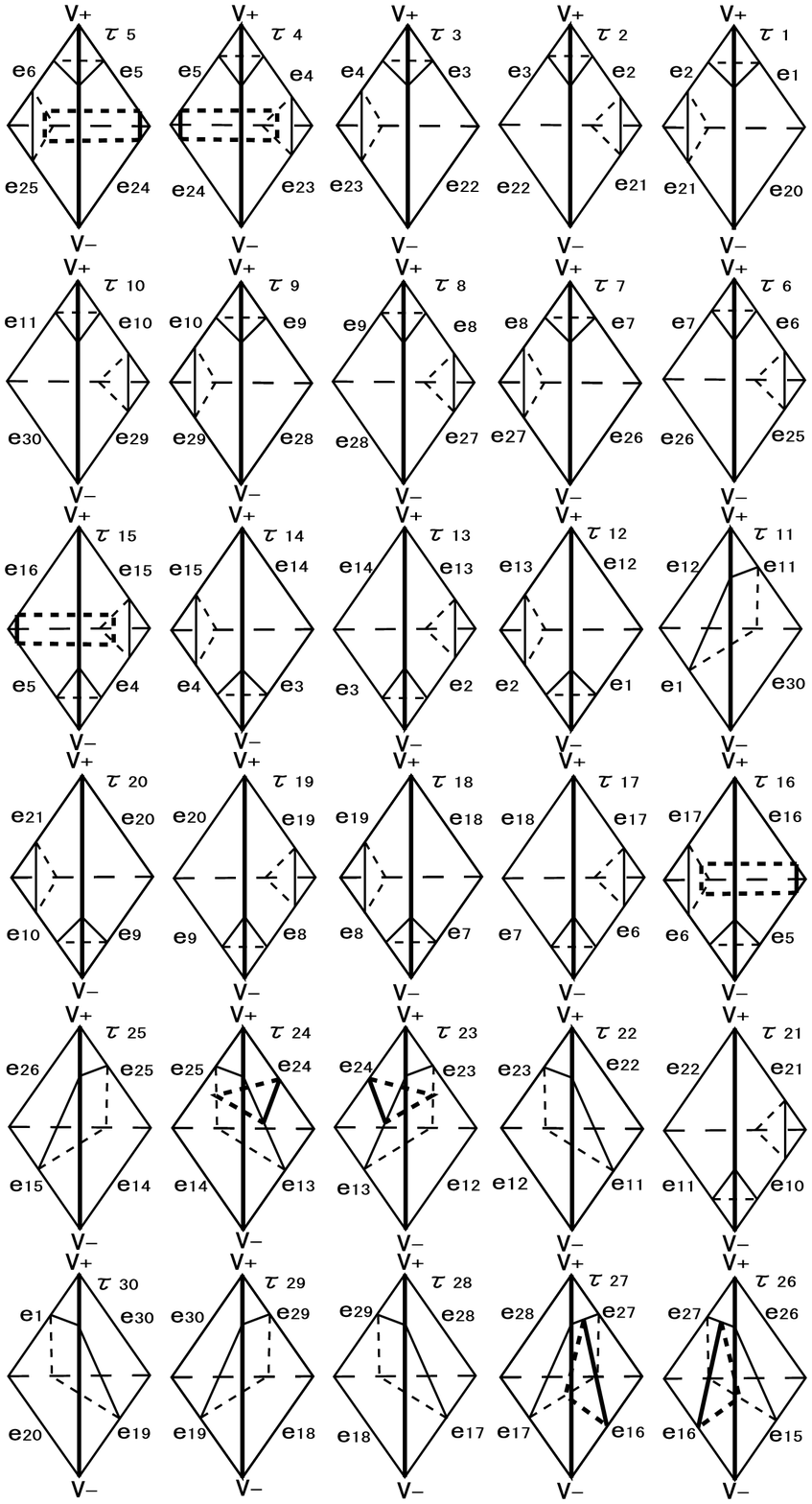}
\end{center}
\caption{}
\label{fig:CprDisk30-11}
\end{figure}



 Let $k$ be an integer equal to or larger than $3$.
 When $n \ge k+1$,
we perform the $k$-th step of compressions after the $(k-1)$-st step ones.
 This operation is performed 
along $(q_{n-(k-1)}-1)/2$ compressing disks.
 The $i$-th compressing disk is composed 
of $q_k + 1$ pairs of disk patches.
 The first pair is composed of two trigonal disk patches
which are in $\tau_{2q_n + 2q_{n-1} + \cdots + 2q_{n-(k-2)} + 2i - 1}$
and $\tau_{2q_n + 2q_{n-1} + \cdots + 2q_{n-(k-2)} + 2i}$
similar as those in Figure \ref{fig:TCprDiskPatch} (a).
 For $j \in \{ 2,3, \cdots, q_k \}$, 
the $j$-th pair consists of quadrilateral disk patches in 
$\tau_{(j+1) \cdot q_n + 2q_{n-1} + \cdots + 2q_{n-(k-2)} + 2i -1} 
\cup 
\tau_{(j+1) \cdot q_n + 2q_{n-1} + \cdots + 2q_{n-(k-2)} + 2i}$
as described in Figure \ref{fig:QCprDiskPatchGeneral}.
 The last pair is composed of trigonal disk patches in 
\newline
$\tau_{(q_k + 2) q_n + 2q_{n-1} + \cdots + 2q_{n-(k-2)} + 2i -1} 
\cup 
\tau_{(q_k + 2) q_n + 2q_{n-1} + \cdots + 2q_{n-(k-2)} + 2i}$
(Figure \ref{fig:TCprDiskPatch} (b)).

\begin{figure}[htbp]
\begin{center}
\includegraphics[width=14cm]{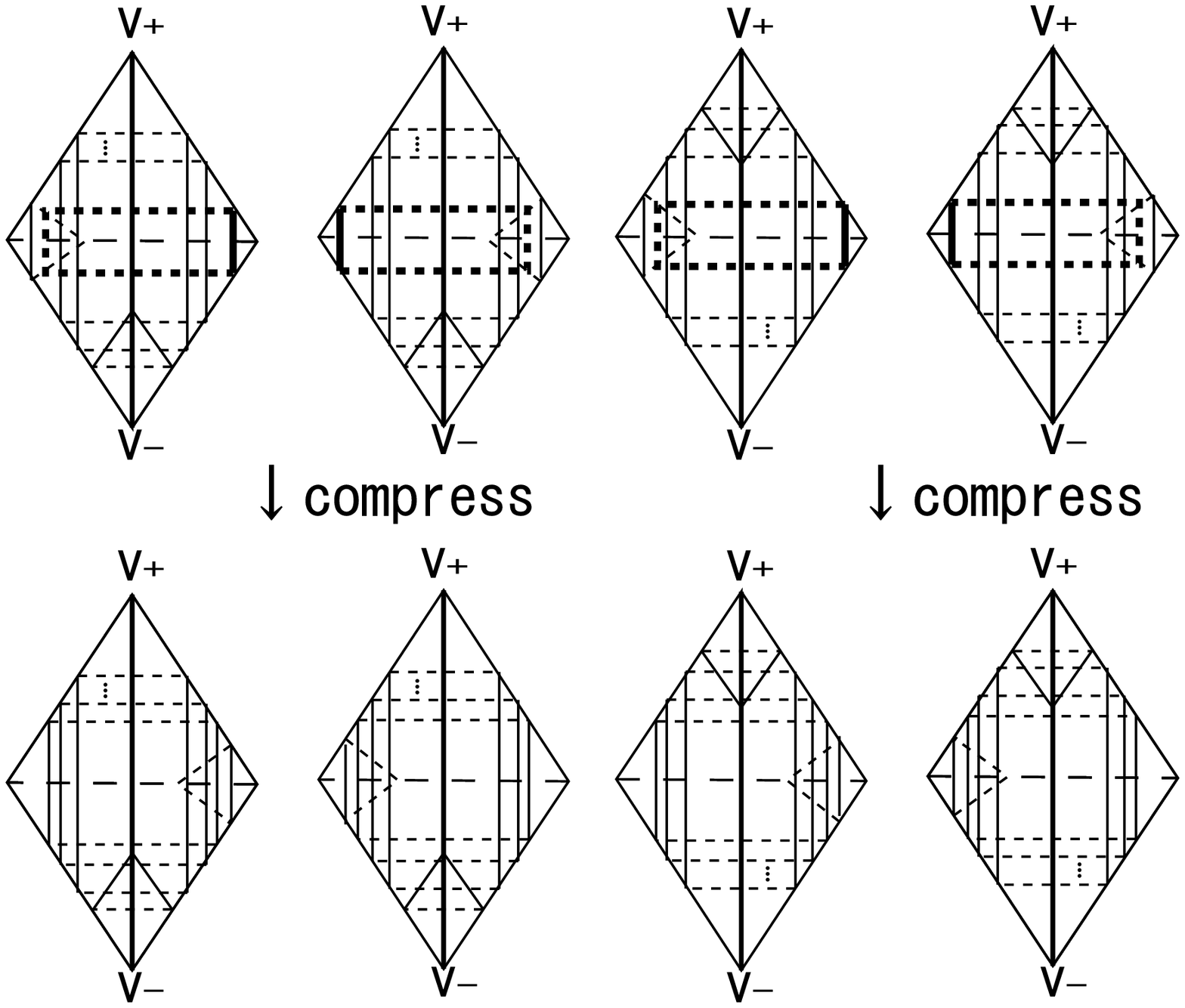}
\end{center}
\caption{}
\label{fig:QCprDiskPatchGeneral}
\end{figure}

 We will prove that we can place disk patches of compression disks as above
in section \ref{section:placement}.
 We should show that the first and the last pairs are in tetrahedra
where the compressing operations of the preceding steps have not occurred,
and the second through the $q_k$-th pairs are in tetrahedra
where such compressing operations have occurred.
 Moreover, for each integer $k$ with $1 \le k \le n-1$,
every tetrahedron must contain
at most one disk patch of compressing disks of the $k$-th step.

 After the $k$-th step, Q-coordinate will be
\newline
${\bf h}_k = 
{\bf h}_{k-1}
-\displaystyle\sum_{i=1}^{(q_{n-(k-1)} -1)/2} 
(\displaystyle\sum_{j=1}^{q_k} {\bf t}_{(j+1) \cdot q_n + 2q_{n-1} + \cdots + 2q_{n-(k-2)} + 2i-1}$
\newline
$\ \ \ \ \ \ \ \ \ \ \ \ \ \ \ \ \ \ \ \ \ \ \ \ \ \ \ \ \ \ \ \ 
- \displaystyle\sum_{j=2}^{q_k} ({\bf s}_{(j+1)q_n + 2q_{n-1} + \cdots + 2q_{n-(k-2)} + 2i -1} 
                  + {\bf s}_{(j+1)q_n + 2q_{n-1} + \cdots + 2q_{n-(k-2)} + 2i}))$
\newline
Subtracting 
$\sum_{j=1}^{q_k} {\bf t}_{(j+1)q_n + 2q_{n-1} + \cdots + 2q_{n-(k-2)} + 2i-1}
- \sum_{j=2}^{q_k} ({\bf s}_{(j+1)q_n + 2q_{n-1} + \cdots + 2q_{n-(k-2)} + 2i -1}$
\newline
$+ {\bf s}_{(j+1)q_n + 2q_{n-1} + \cdots + 2q_{n-(k-2)} + 2i})$
corresponds to the surgery along the $i$-th compressing disk.

 A single compression increases the Euler characteristic of the surface by two.
 Hence we obtain a closed surface ${\bf h}_{k}$
of Euler characteristic $2 - p_n /2 + \sum_{r=1}^{k} (q_{n-(r-1)} -1)$.
 The resulting surface is non-orientable 
and fundamental with respect to Haken's matching equations
by Lemmas \ref{lemma:orientability} and \ref{lemma:HakenFund}.


 After the compressing operations of the $(n-1)$-st step,
we obtain a fundamental closed non-orientable surface 

\noindent
${\bf h}_{n-1} = 
(\displaystyle\sum_{m=1}^{p/2} {\bf t}_{2m-1})/2$
\newline
$\ \ \ \ \ \ \ \ \ 
-\displaystyle\sum_{k=1}^{n-1}
\displaystyle\sum_{i=1}^{(q_{n-(k-1)} -1)/2} 
(\displaystyle\sum_{j=1}^{q_k} {\bf t}_{(j+1)q_n + 2q_{n-1} + \cdots + 2q_{n-(k-2)} + 2i-1}$
\newline
$\ \ \ \ \ \ \ \ \ \ \ \ \ \ \ \ \ \ \ \ \ \ \ \ \ \ \ \ \ \ \ \ \ \ \ \ \ \ \ 
-\displaystyle\sum_{j=2}^{q_k} ({\bf s}_{(j+1)q_n + 2q_{n-1} + \cdots + 2q_{n-(k-2)} + 2i -1} 
                              + {\bf s}_{(j+1)q_n + 2q_{n-1} + \cdots + 2q_{n-(k-2)} + 2i}))$

\noindent
where 
$\sum_{j=1}^{q_k} {\bf t}_{(j+1)q_n + 2q_{n-1} + \cdots + 2q_{n-(k-2)} + 2i-1}
- \sum_{j=2}^{q_k} ({\bf s}_{(j+1)q_n + 2q_{n-1} + \cdots + 2q_{n-(k-2)} + 2i -1}$
\newline
$+ {\bf s}_{(j+1)q_n + 2q_{n-1} + \cdots + 2q_{n-(k-2)} + 2i})$
is defined to be 
${\bf t}_{2i-1}$
when $k=1$,
and
$\sum_{j=1}^{q_2} {\bf t}_{(j+1)q_n + 2i-1}
-\sum_{j=2}^{q_2} ({\bf s}_{(j+1)q_n + 2i -1} 
                 + {\bf s}_{(j+1)q_n + 2i})$
when $k=2$.
 This surface ${\bf h}_{n-1}$ is of Euler characteristic
$2 - p_n /2 + \sum_{k=1}^{n-1} (q_{n-(k-1)} -1) = 2-n$
by the formula Lemma \ref{lemma:formulae} (2).
 This is the maximal Euler characteristic
among all closed non-orientable surfaces
in the $(p_n, q_n)$-lens space,
which we will show in section \ref{section:euler}.

 The surface ${\bf h}_{n-1}$ has 
$n-2$ sheets of normal disks of type $X_{m1}$ in $\tau_m$
for $m = \sum_{u=1}^{n-1} q_u$, $(\sum_{u=1}^{n-1} q_u) + 1$,
$\sum_{u=1}^n q_u$ and $(\sum_{u=1}^n q_u) + 1$.
 We will show this for $m=\sum_{u=1}^{n-1} q_u$
in section \ref{section:multiple}.


\section{Formulae for $p_n$ and $q_n$}\label{section:formulae}

 We give some formulae for $p_n$ and $q_n$ in this section.

\begin{lemma}\label{lemma:formulae}
 Let $\{ p_n \}$ and $\{ q_n \}$ be infinite sequences of integers
defined by 
$p_0=0$, $q_0=1$,
$p_n = (\kappa+1) p_{n-1} + \kappa q_{n-1}$
and $q_n = p_{n-1} + q_{n-1}$,
where $\kappa$ is a natural number.
 Then the formulae below hold.
\begin{enumerate}
\item[(1)] $p_n = \kappa q_n + p_{n-1}$
\item[(2)] $\kappa(q_1 + q_2 + \cdots + q_{\ell}) = p_{\ell}$ for $\ell \ge 1$.
\item[(3)] 
$-(q_1 + q_2 + \cdots + q_{\ell-1})p_n + q_{\ell} q_n
=-(q_1 + q_2 + \cdots + q_{\ell-2})p_{n-1} + q_{\ell-1} q_{n-1}$,
where $\ell \ge 2$ and $n \ge 1$.
When $\ell = 2$,
we define the sum $q_1 + q_2 + \cdots + q_{\ell-2}$ to be equal to $0$.
\item[(4)] 
$-(q_1 + q_2 + \cdots + q_{m-1})p_n + q_m q_n = q_{n-(m-1)}$,
and hence $q_m q_n \equiv q_{n-(m-1)}$ {\rm (mod. $p_n$)}
for $n \ge m-1$.
\item[(5)]
$(\kappa+1)q_n-p_n=q_{n-1}$ for $n \ge 1$.
\item[(6)]
$(2\kappa+1)q_n-2p_n =-p_{n-1}+q_{n-1}$ for $n \ge 1$.
\end{enumerate}
\end{lemma}

\begin{proof}
(1) 
$p_n = (\kappa+1)p_{n-1} + \kappa q_{n-1}
= \kappa (p_{n-1} + q_{n-1}) + p_{n-1}
= \kappa q_n + p_{n-1}$

(2) We prove this formula by induction on $\ell$.
\newline
$\kappa(q_1 + q_2 + \cdots + q_{\ell}) 
= \kappa (q_1 + q_2 + \cdots + q_{\ell-1}) + \kappa q_{\ell}
= p_{\ell-1} + \kappa q_{\ell}
= p_{\ell}$

The second equality holds by the assumption of induction.
The last equality is the formula (1).

(3)
$-(q_1 + q_2 + \cdots + q_{\ell-1}) p_n + q_{\ell} q_n
=-(q_1 + q_2 + \cdots + q_{\ell-1})(\kappa q_n + p_{n-1}) + q_{\ell} q_n$
\newline
$=-\kappa (q_1 + q_2 + \cdots + q_{\ell-1}) q_n 
 - (q_1 + q_2 + \cdots + q_{\ell-1}) p_{n-1} + q_{\ell} q_n$
\newline
$=-p_{\ell-1} q_n - (q_1 + q_2 + \cdots + q_{\ell-1}) p_{n-1} + q_{\ell} q_n$
\newline
$=-p_{\ell-1} q_n - (q_1 + q_2 + \cdots + q_{\ell-1}) p_{n-1} 
 + (p_{\ell-1} + q_{\ell-1}) q_n$
\newline
$=- (q_1 + q_2 + \cdots + q_{\ell-1}) p_{n-1} + q_{\ell-1}q_n$
\newline
$=- (q_1 + q_2 + \cdots + q_{\ell-2}) p_{n-1} 
 - q_{\ell-1} p_{n-1} + q_{\ell-1}(p_{n-1} + q_{n-1})$
\newline
$=- (q_1 + q_2 + \cdots + q_{\ell-2}) p_{n-1} + q_{\ell-1} q_{n-1}$

The first equality holds by (1).
The third equality follows from (2).

(4) Applying (3) $m-1$ times we have:
\newline
$-(q_1 + q_2 + \cdots + q_{m-1})p_n + q_m q_n 
=-(q_1 + q_2 + \cdots + q_{m-2})p_{n-1} + q_{m-1} q_{n-1}$ 
\newline
$=-(q_1 + q_2 + \cdots + q_{m-3})p_{n-2} + q_{m-2} q_{n-2} 
= \cdots $
\newline
$=-(q_1 + q_2 + \cdots + q_{m-m})p_{n-(m-1)} + q_{m-(m-1)} q_{n-(m-1)} 
= 0 + q_1 q_{n-(m-1)} = q_{n-(m-1)}$.
%

(5) 
$(\kappa+1)q_n-p_n
=(\kappa+1)(p_{n-1}+q_{n-1})-((\kappa+1)p_{n-1}+\kappa q_{n-1}) 
= q_{n-1}$.

(6) 
$(2\kappa+1)q_n-2p_n
=(2\kappa+1)(p_{n-1}+q_{n-1})-2((\kappa+1)p_{n-1}+\kappa q_{n-1})
=-p_{n-1}+q_{n-1}$.
\end{proof}


\section{placement of disk patches}\label{section:placement}

 In this section,
we show that we can place the disk patches of the compression disks
as in section \ref{section:construction}.


\begin{lemma}\label{lemma:LastRegion}
 Let $n$ , $k$ and $r$ be integers
with $n \ge 3$, $2 \le k \le n-1$ and $1 \le r \le p_{n-1}$.

\begin{enumerate}
\item[(1)] 
 The first pair of disk patches 
and the last pair of disk patches 
of the compression disks in section \ref{section:construction}
are contained in tetrahedra
in the last region.
\item[(2)]
 The tetrahedron $\tau_{2q_n+r}$ 
in the last region of the $(p_n, q_n)$-lens space 
contains the leading (resp. following)
disk patch of the $j$-th pair
of the $i$-th compressing disk of the $k$-th step
if and only if 
$\tau_r$ of the $(p_{n-1}, q_{n-1})$-lens space
contains the leading (resp. following) disk patch of the $j$-th pair
of the $i$-th compressing disk of the $(k-1)$-st step.
\end{enumerate}
\end{lemma}

 The number of compressing disks of the $k$-th step
for $(p_n, q_n)$-lens space
is $(q_{n-(k-1)}-1)/2$,
and that of the $(k-1)$-st step
for $(p_{n-1}, q_{n-1})$-lens space
is $(q_{(n-1)-((k-1)-1)}-1)/2$.
 Note that they are equal.

\begin{proof}
 The leading disk patch 
of the $1$st pair of the $i$-th compressing disk of the $k$-th step 
for the $(p_n, q_n)$-lens space
is contained in 
$\tau_{2(q_n + q_{n-1} + \cdots + q_{n-(k-2)}) + 2i-1}$.
 This tetrahedron is contained in the last region,
which follows from Lemma \ref{lemma:formulae} (2)
and the fact that the number of compression disks of the $k$-th step 
is $(q_{n-(k-1)}-1)/2$.
 On the other hand,
$\tau_{2(q_{n-1} + q_{n-2} + \cdots + q_{(n-1)-((k-1)-2)}) + 2i-1}$
of the $(p_{n-1}, q_{n-1})$-lens space
contains the leading disk patch 
of the $1$st pair of the $i$-th compressing disk of the $(k-1)$-st step.
 The numbers assigned to these two tetrahedra differ by $2q_n$.

 The leading disk patch 
of the last pair of the $i$-th compressing disk of the $k$-th step 
for the $(p_n, q_n)$-lens space
is contained in the tetrahedron numbered
\newline
$2(q_n + q_{n-1} + \cdots + q_{n-(k-2)}) + 2i-1 + q_n q_k
\equiv 
2(q_n + q_{n-1} + \cdots + q_{n-(k-2)}) + 2i-1 + q_{n-(k-1)}$
\newline
(mod. $p_n$) since the last pair is the $(q_k + 1)$-st one.
 The congruence $\equiv$ follows by Lemma \ref{lemma:formulae} (4).
 This tetrahedron is in the last region
by a similar argument as above.
 On the other hand,
the leading disk patch 
of the last pair of the $i$-th compressing disk of the $(k-1)$-st step 
for the $(p_{n-1}, q_{n-1})$-lens space
is contained in the tetrahedron numbered
\newline
$2(q_{n-1} + q_{n-2} + \cdots + q_{(n-1)-((k-1)-2)}) + 2i-1 + q_{n-1} q_{k-1}$
\newline
$\equiv 
2(q_{n-1} + q_{n-2} + \cdots +q_{(n-1)-((k-1)-2)}) + 2i-1 +q_{(n-1)-((k-1)-1)}$.
(mod. $p_{n-1}$)
\newline
 The numbers assigned to these two tetrahedra differ by $2q_n$.

 Similar things hold for the following disk patches
of the first and the last pairs
since they are in tetrahedra
next to those contain the leading disk patches.

 Suppose that
$\tau_{2q_n + r}$ in the last region contains
the leading (resp. following) disk patch of the $j$-th pair 
of the $i$-th compressing disk of the $k$-th step
for the $(p_n, q_n)$-lens space.
 Let $j'$ be the number assigned to the next pair
whose leading (resp. following) disk patch is contained 
in a tetrahedron $\tau_{2q_n + r'}$ in the last region.
 Then we will show that 
either $r'-r = q_{n-1}$ (when $1 \le r \le p_{n-1} -q_{n-1}$)
or $r' -r = q_{n-1} - p_{n-1}$ (when $p_{n-1} -q_{n-1}< r \le p_{n-1}$)
holds
as if these two leading (resp. following) disk patches were of adjacent pairs
of a compressing disk for $(p_{n-1}, q_{n-1})$-lens space.

 If $1 \le r \le p_{n-1} - q_{n-1}$,
then the leading (resp. following) disk patches of the $(j+1)$-st pair,
the $(j+2)$-nd pair
and the $(j+3)$-rd pair
are contained 
in a tetrahedron in the first region,
the second region
and the last region
respectively.
 In fact,
the leading (resp. following) disk patch
of the $(j+\ell)$-th pair
is contained in 
$\tau_{(\ell+2)q_n + r}=\tau_{(\ell+2)q_n + r -p_n}$,
and 
$(1<)q_{n-1}+1 \le 3q_n + r - p_n \le p_{n-1}(<q_n)$,
$q_n+q_{n-1}+1 \le 4q_n + r - p_n \le q_n + p_{n-1}$
and
$2q_n+q_{n-1}+1 \le 5q_n + r - p_n \le 2q_n + p_{n-1}$
hold by Lemma \ref{lemma:formulae} (5)
and $q \le r \le p_{n-1}-q_{n-1}$.
 Hence $j'=j+3$, $r'=3q_n + r - p_n$,
and $r'-r= 3q_n -p_n = q_{n-1}$
as expected.

 If $p_{n-1} - q_{n-1} < r \le p_{n-1}$,
then the leading (resp. following) disk patch of the $(j+1)$-st pair,
the $(j+2)$-nd pair,
the $(j+3)$-rd pair
the $(j+4)$-th pair,
and the $(j+5)$-th pair
are contained 
in a tetrahedron in the first region,
the second region,
the first region,
the second region
and the last region
respectively.
 In fact,
the leading (resp. following) disk patch
of the $(j+\ell)$-th pair
is contained in 
$\tau_{(\ell+2)q_n + r}
=\tau_{(\ell+2)q_n + r -p_n}
=\tau_{(\ell+2)q_n + r -2p_n}$,
and 
$(1<)p_{n-1} < 3q_n+r-p_{n-1} \le q_n$,
$q_n+p_{n-1} < 4q_n + r - p_n \le 2q_n$,
$0 < 5q_n + r - 2p_n \le q_{n-1} (<q_n)$,
$q_n < 6q_n + r - 2p_n \le q_n + q_{n-1}$,
$2q_n < 7q_n + r - 2p_n \le 2q_n + q_{n-1}(<p_n)$
hold by Lemma \ref{lemma:formulae} (5), (6)
and $p_{n-1}-q_{n-1} < r \le p_{n-1}$.
 Hence $j'=j+5$, $r'=5q_n + r - 2p_n$,
and $r'-r= 5q_n - 2p_n = -p_{n-1} + q_{n-1}$
as expected.

 The above arguments show the lemma
since $p_{n-1}$ and $q_{n-1}$ are coprime,
since the number of pairs of a compressing disk of the $k$-th step 
is $q_k + 1$
and since this number is smaller than $p_{n-1}$ 
because of the conditions $k \le n-1$ and $n \ge 3$.
\end{proof}

\begin{lemma}\label{lemma:LastTetrahedra}
\begin{enumerate}
\item[(1)]
 No disk patch is placed 
in $\tau_{q_n}$, $\tau_{2q_n}$ and $\tau_{p_n}$,
the last tetrahedra of the first, the second and the last regions.
\item[(2)]
 For any integer $j$ with $1 \le j \le p_k + 1$
and any integer $k$ with $1 \le k \le n-1$,
the $j$-th pairs of disk patches of 
$(q_{n-(k-1)}-1)/2$ compressing disks of the $k$-th step
is in the same region.
\end{enumerate}
\end{lemma}

\begin{proof}
 As shown in Figure \ref{fig:ComprDisk8-3},
the last tetrahedron $\tau_{p_n}$ does not contain a disk patch
for $n=2$.
 Then Lemma \ref{lemma:LastRegion} (2) and an inductive argument
gurantee that the last tetrahedron does not contain a disk patch.
 Since every compressing disk has the first pair of disk patches
and the last pair of disk patches in the last region
as shown in Lemma \ref{lemma:LastRegion} (1),
no disk patches are in $\tau_{q_n} \cup \tau_{2q_n}$.
 Thus (1) holds.

 Then (2) follows
because
the $j$-th pairs of disk patches of 
$(q_{n-(k-1)}-1)/2$ compressing disks of the $k$-th step
are contained in $q_{n-(k-1)}-1$ consecutive tetrahedra.
\end{proof}

\begin{lemma}
\begin{enumerate}
\item[(1)] 
The first and the last pairs of disk patches of a compressing disk 
in section \ref{section:construction}
are in tetrahedra
where the compressing operations of the preceding steps have not occurred.
 The second through the $q_k$-th pairs of disk patches of a compressing disk
in section \ref{section:construction}
are in tetrahedra
where such compressing operations have already occurred.
\item[(2)]
 It does not occur
that a disk patch of the $j_1$-th pair of the $i_1$-th compressing disk
of the $k$-th step
and that of the $j_2$-th pair of the $i_2$-th compressing disk
of the $k$-th step
with $j_1 \ne j_2$ or $i_1 \ne i_2$
are in the same tetrahedron.
\end{enumerate}
\end{lemma}

\begin{proof}
 We show this lemma by an inductive argument on $n$.
 For $n=2$, this lemma holds as shown in Figure \ref{fig:ComprDisk8-3}.
 We assume it holds for the $(p_{n-1}, q_{n-1})$-lens space.

 For $(p_n, q_n)$-lens space, this lemma holds in the last region
by Lemma \ref{lemma:LastRegion}
and an inductive argument.


 In the first and the second regions,
disk patches are placed in the tetrahedra
where compressing operations of the first step have already occured
by Lemma \ref{lemma:LastTetrahedra} (1).
 Moreover,
Lemma \ref{lemma:LastTetrahedra} (2) implies
that (2) of this lemma holds for disk patches in the first and the second regions
because
(2) of this lemma holds for disk patches in the last region
where the first and the last patches are.
\end{proof}


\section{maximality of Euler characteristic}\label{section:euler}

 In this section,
we show that the surface ${\bf h}_{n-1}$
constructed in section \ref{section:construction}
is of maximal Euler characteristic
among all closed non-orientable surfaces
in the $(p_n, q_n)$-lens space.

 A formula of the minimal crosscap number
among those of all such surfaces
is given by Bredon and Wood in \cite{BW}.
 We briefly recall it.
 We consider a $(p,q)$-lens space with $p$ even.
 $p/q$ can be presented as a continued fraction:
\newline
$p/q 
= 
[a_0, a_1, \cdots, a_m]
=
a_0 + 
\dfrac{1}{a_1 + 
\dfrac{1}{
\begin{array}{c}
a_2 + \\
 \\
 \\
\end{array}
\begin{array}{c}
 \\
\ddots \\
 \\
\end{array}
\begin{array}{c}
 \\
 \\
+ \dfrac{1}{a_m} \\
\end{array}
}}$
\newline
where $a_i$ are integers, 
$a_0 \ge 0,\ a_i >0$ for $1 \le i \le m$, and $a_m > 1$.
 We define $b_0, b_1, \cdots, b_m$, inductively,
\newline
$b_0 = a_0$,
\newline
$b_i = 
\left\{ \begin{array}{l}
a_i \ \ {\rm if}\ b_{i-1} \ne a_{i-1}\ {\rm or\ if}\ \displaystyle\sum_{j=0}^{i-1} b_j\ {\rm is\ odd} \\
0   \ \ {\rm if}\ b_{i-1} = a_{i-1}\ {\rm and}\ \displaystyle\sum_{j=0}^{i-1} b_j\ {\rm is\ even.} \\
\end{array} \right.$
\newline
Then, 
the minimal crosscap number
is $x_{\rm max}(p,q) = \dfrac{1}{2} \displaystyle\sum_{i=0}^m b_i$.

 For our case of $(p_n, q_n)$-lens space,
\newline
$\dfrac{p_1}{q_1} = 2$, 
$\ \ \dfrac{p_2}{q_2}=
2+\dfrac{1}{
1+\dfrac{1}{2}},\ \ \cdots, $
$\ \ \dfrac{p_n}{q_n}=
2+\dfrac{1}{
1+\dfrac{1}{
2+\dfrac{1}{
\begin{array}{c}
1+ \\
 \\
 \\
\end{array}
\begin{array}{c}
 \\
\ddots \\
 \\
\end{array}
\begin{array}{c}
 \\
 \\
+ \dfrac{1}{2} \\
\end{array}
}}}$
\newline
and hence, for $0 \le i \le 2(n-1)$,
$a_i = 
\left\{ \begin{array}{l}
2\ ({\rm when}\ i\ {\rm is\ even}) \\
1\ ({\rm when}\ i\ {\rm is\ odd}) \\
\end{array} \right.$
and 
$b_i = 
\left\{ \begin{array}{l}
2\ ({\rm when}\ i\ {\rm is\ even}) \\
0\ ({\rm when}\ i\ {\rm is\ odd}). \\
\end{array} \right.$
Thus $x_{\rm min}(p_n, q_n) = n$,
and the maximal Euler characteristic is $2-n$.


\section{$n-2$ parallel sheets of normal disks}\label{section:multiple}

 In this section, 
we will show the next lemma,
which implies that 
the compressing operations in section \ref{section:construction} occur 
$n-1$ times in $\tau_m$ with $m=\sum_{u=1}^{n-1} q_u$,
and hence $\tau_m$ contains $n-2$ sheets of normal disk of type $X_{m1}$.

\begin{lemma}
 For an odd integer $k$ with $1 \le k \le n-1$,
the leading disk patch
of the $(q_{n-k}-\sum_{j=1}^{n-(k+1)} q_j)$-th pair
of the $(((\sum_{r=1}^k q_r)+1)/2)$-th compression disk
of the $(n-k)$-th step
is in $\tau_m$ with $m=\sum_{u=1}^{n-1} q_u$.
 For an even integer $k$ with $1 \le k \le n-1$,
the following disk patch
of the $(q_{n-k}-\sum_{j=1}^{n-(k+1)} q_j)$-th pair
of the $((\sum_{r=1}^k q_r)/2)$-th compression disk
of the $(n-k)$-th step
is in $\tau_m$ with $m=\sum_{u=1}^{n-1} q_u$.
\end{lemma}

 Disk patches as in the above lemma
actually appear in the construction in section \ref{section:construction}.
 In fact, using the formula of Lemma \ref{lemma:formulae} (2),
we have 
$q_1 + q_2 + \cdots + q_{\ell} +1
= \dfrac{p_{\ell}}{2} +1 
\le p_{\ell}
< p_{\ell}+q_{\ell}
= q_{\ell+1}$
for $\ell \ge 1$.
 This implies
$1 \le q_{n-k}-\sum_{j=1}^{n-(k+1)} q_j \le q_{n-k}+1$
and 
$1 \le ((\sum_{r=1}^k q_r)+1)/2 \le (q_{n-((n-k)-1)}-1)/2$.


\begin{proof}
 For an odd integer $k$,
the leading disk patch is contained in the tetrahedron numbered
\newline
$2(q_n + q_{n-1} + \cdots + q_{n-((n-k)-2)}) + 
2(((\sum_{r=1}^k q_r)+1)/2) -1) + 
q_n (q_{n-k}-(\sum_{j=1}^{n-(k+1)} q_j) -1) + 
1$
\newline
$=
2(q_n + q_{n-1} + \cdots + q_{k+2}) + 
(\sum_{r=1}^k q_r)+1-2 +
q_n q_{n-k} - (\sum_{j=1}^{n-(k+1)} q_n q_j) -q_n +
1$
\newline
$\equiv
2(q_n + q_{n-1} + \cdots + q_{k+2}) + 
(\sum_{r=1}^k q_r)+1-2 +
q_{k+1}-(\sum_{j=1}^{n-(k+1)} q_{n-j+1}) -q_n +
1$
\newline
$=
q_1 + q_2 + \cdots + q_{n-1}$.
\newline
where the symbol $\equiv$ denotes congruence modulo $p_n$,
and it holds by the formula of Lemma \ref{lemma:formulae} (4).

 The proof is similar for an even integer $k$, and we omit it.
\end{proof}


\bibliographystyle{amsplain}

\medskip

\noindent
Miwa Iwakura, Chuichiro Hayashi:
Department of Mathematical and Physical Sciences,
Faculty of Science, Japan Women's University,
2-8-1 Mejirodai, Bunkyo-ku, Tokyo, 112-8681, Japan.
hayashic@fc.jwu.ac.jp

\end{document}